\documentclass[11pt]{article}
\usepackage{amssymb,amsmath,amsthm}
\newcommand{\cl}{\mathop{\rm cl}}
\newcommand{\fr}{\mathop{\rm fr}}
\newcommand{\inter}{\mathop{\rm Int}}

\newcommand{\dist}{\mathop{\rm dist}}
\newcommand{\mes}{\mathop{\rm mes}}
\newcommand{\Id}{\mathop{\rm Id}}

\newcommand{\contg}{\mathop{\rm contg}}

\newcommand{\capac}{\mathop{\rm capac}}
\renewcommand{\Im}{\mathop{\rm Im}}

\newtheorem{lemma}{Lemma}[section]
\newtheorem{theorem}{Theorem}[section]
\newtheorem{corollary}{Corollary}[section]

\theoremstyle{definition}

\theoremstyle{remark}

\theoremstyle{theorem}

\title{Unique Determination of Polyhedral Domains in $\mathbb R^n$
($n \ge 4$) and $p$-Moduli of Path Families\footnote{
$p$-modulus
of path families, boundary condenser, quasiconformal and conformal
mappings, isometric mapping, unique determination}}
\author{Anatoly~P.~Kopylov\footnote{Sobolev Institute of Mathematics,
Acad. Koptyuga pr. 4, and Novosibirsk State University, Pirogova
str., 2, 630090 Novosibirsk, Russia; apkopylov@yahoo.com}}
\date{01.01.2014}

\begin{document}

\maketitle

%\section{Abstract}\label{s0}
\begin{abstract}

This paper is an extension of the author's lecture ``Unique
Determination of Polyhedral Domains and $p$-Moduli of Path Families"
given at the International Conference ``Metric
Geometry of Surfaces and Polyhedra'' dedicated to the 100th anniversary
of Prof. Nikolay Vladimirovich Efimov,
which was held in Moscow (Russia) in August 2010 (in this
connection, see, for example,~\cite{Ko}).
We expose new results on the problem of the
unique determination of conformal type for domains in~$\mathbb R^n$.
It is in particular established that a (generally speaking) nonconvex
bounded polyhedral domain in
$\mathbb R^n$ ($n \ge 4$)
whose boundary is an
$(n - 1)$-dimensional
connected manifold of class
$C^0$
without boundary and can be represented as a finite union of
pairwise nonoverlapping
$(n-1)$-dimensional
cells is uniquely determined by the relative conformal moduli of
its boundary condensers.

Results on the unique determination (of polyhedral domains)
of isometric type are also obtained. In contrast to the
classical case, these results present a new approach in which
the notion of the $p$-modulus
of path families is used.

\end{abstract}

%\medskip

%\tableofcontents

\section{Introduction}\label{s1}

In development of the classical topic of the unique
determination of closed convex surfaces by their intrinsic
metrics~\cite{Po}, in~\cite{Ko1}-\cite{Ko3} the author started
a search for a complete description of the boundary values of
conformal mappings of domains in the space
$\bar{\mathbb R}^n$
(by
$\bar{\mathbb R}^n$
we denote the one-point compactification
$\mathbb R^n \cup \infty$
of the real Euclidean space
$\mathbb R^n$).
This search is based on the notion of the
$n$-modulus
of a family of curves first introduced in~\cite{AB}, which plays
a very important role in various domains of mathematics. In
particular, using the notion of the modulus of a family of
curves, one can obtain the following characterization of conformal
mappings~\cite{V} (see also~\cite{Ge1},~\cite{Re}):
A homeomorphism
$f: U \to V$
of domains
$U$
and
$V$
in $\bar{\mathbb R}^n$ ($n \ge 2$)
is conformal if and only if every family of curves
$\Gamma$
in
$U$
satisfies the condition
\begin{equation}\label{eq1.1}
M_n(\Gamma') = M_n(\Gamma),
\end{equation}
where
$\Gamma' = \{f\circ\gamma: \gamma \in \Gamma\}$
(in other words, a mapping
$f$
is conformal if and only if it is
$1$-quasiconformal).

Further, suppose that the boundaries
$\fr U$
and
$\fr V$
of two domains
$U$
and
$V$
are sufficiently regular (e.g., they are bounded and are
Lipschitz manifolds of dimension
$n - 1$
without boundary). Then any quasiconformal mapping of these
domains can be extended to a homeomorphism
$H: \cl U \to \cl V$
of their closures
$\cl U$
and
$\cl V$\cite{V};
moreover, in the case of conformal mappings, the extension
$H$
satisfies~(\ref{eq1.1}) as well. In particular,~(\ref{eq1.1})
holds for the
$n$-moduli
of the families
$\Gamma$
of paths joining in
$U$
the components
$F_1$
and
$F_2$
of the boundary condensers
$F = \{F_1,F_2\}$
of the domain
$U$
(in this case,
$f = H|_{\fr U}$
in~(\ref{eq1.1})).
This gives rise to the natural question: Are domains
$U$
and
$V$
conformally equivalent if the boundary of one of them can be
mapped onto the boundary of the other by means of a homeomorphism
$f: \fr U \to \fr V$
preserving the
$n$-moduli
of the families of paths joining the components of boundary
condensers in the domain
$U$?

In~\cite{Ko1}-\cite{Ko3}, we gave a positive answer to this question in the case
of convex domains. Namely, therein, we proved the following theorem:

\begin{theorem}\label{t1.1}
If
$n \ge 4$
then any bounded convex polyhedral domain
$U \subset \mathbb R^n$
{\rm(}i.e.{\rm,} a nonempty bounded intersection of finitely
many open
$n$-dimensional
half-spaces{\rm)} is uniquely determined by the relative
conformal moduli of its boundary condensers in the class
$\mathcal P$
of all bounded convex polyhedral domains
$V \subset \mathbb R^n$.
\end{theorem}

This paper continues the study of unique determination of
conformal type initiated in~\cite{Ko1}-\cite{Ko3} by
Theorem~\ref{t1.1}. First, we briefly recall notions
from~\cite{Ko1}-\cite{Ko3} used in Theorem~\ref{t1.1} that we will
need in the sequel.

Let
$U \subset \mathbb R^n$
($U \ne \mathbb R^n$)
be a domain in~$\mathbb R^n$
for which
$\fr U$
is a Lipschitz
$(n - 1)$-manifold
without boundary. A boundary condenser
$F = \{F_1,F_2\}$
of~$U$
is a pair of disjoint closed subsets
$F_1$
and
$F_2$
of the boundary
$\fr U$
of this domain (at least one of which is bounded). A relative
conformal modulus
$M^U(F)$
of a boundary condenser
$F$
of the domain
$U$
is by definition the $n$-modulus
\begin{equation}\label{eq1.2}
M_n(\Gamma_{F_1,F_2,U}) = \inf_{\rho \in \mathcal R(\Gamma_{F_1,F_2,U})}
\int_{\mathbb R^n}[\rho(x)]^n dx
\end{equation}
of the family
$\Gamma_{F_1,F_2,U}$
of all continuous paths
$\gamma: [0,1] \to \cl U$,
where
$\cl U$
denotes the closure of~$U$,
such that
$\gamma(0) \in F_1$,
$\gamma(1) \in F_2$,
and
$\gamma(t) \in U$
for
$0 < t < 1$
(in~(\ref{eq1.2}),
$\mathcal R(\Gamma_{F_1,F_2,U})$
is the set of all nonnegative Borel measurable functions
$\rho: \mathbb R^n \to \dot{\mathbb R}$,
where
$\dot{\mathbb R} = \mathbb R \cup \{-\infty,\infty\}$
is the two-point compactification of the real line
$\mathbb R = \mathbb R^1$,
satisfying the condition
$\int_{\gamma}\rho ds \ge 1$
for every rectifiable path
$\gamma \in \Gamma_{F_1,F_2,U}$).

Let
$\mathcal L_0 = \mathcal L_0(n)$
be a subclass in the class
$\mathcal L = \mathcal L(n)$
of all domains
$U$
in~$\mathbb R^n$
with
$n \ge 3$
different from
$\mathbb R^n$
and such that the boundary of each of these domains is a
Lipschitz
$(n - 1)$-manifold
without boundary. Following~\cite{Ko1}-\cite{Ko3}, we say that a
domain
$U \in \mathcal L_0$
is {\it uniquely determined by the relative conformal moduli of its
boundary condensers in the class
$\mathcal L_0$}
if the following conditions hold: Suppose that
$V \in \mathcal L_0$
and there exists a homeomorphism
$f: \fr V \to \fr U$
of the boundary
$\fr V$
of the domain
$V$
to the boundary
$\fr U$
of~$U$
preserving the relative conformal moduli of the boundary
condensers, i.e., such that
$M^V(F) = M^U(f(F))$
(where
$f(F) = \{f(F_1),f(F_2)\}$)
for each boundary condenser
$F$
of~$V$.
Then~$V$ can be mapped conformally onto~$U$.

In connection with Theorem~\ref{t1.1}, there arises the question of
whether the convexity condition in its statement is substantial.
The main results of this paper are Theorems~\ref{t2.1} and~\ref{t2.2}
below, which make it possible to waive the convexity condition
in Theorem~\ref{t1.1}.

The second part of the article is devoted to a complete description
of the boundary values of isometric mappings of $n$-dimensional
domains in terms of the $p$-moduli of path families. In this
connection, we briefly recall now some facts of the theory of
quasi-isometric mappings that we will need below.

\textbf{Definition~1.1.}
Let
$K \in [1,\infty[$.
A homeomorphism
$f: U_1 \to U_2$
of domains
$U_1$
and
$U_2$
in~$\mathbb R^n$
is called $K$-quasi-isometric if
$$
K^{-1} \le \liminf_{y \to x}\frac{|f(y) - f(x)|}{|y - x|} \le
\limsup_{y \to x}\frac{|f(y) - f(x)|}{|y - x|} \le K
$$
for any $x \in U_1$. A homeomorphism
$f: U_1 \to U_2$
is called quasi-isometric if it is
$K$-quasi-isometric for some
$K \in [1,\infty[$.

We have

\begin{theorem}\label{t1.2}
Suppose that
$f: U_1 \to U_2$
is a
$K$-quasi-isometric
homeomorphism of bounded domains
$U_1$
and
$U_2$
in
$\mathbb R^n,$
where
$n \ge 2$
$(1 \le K < \infty)$.
Then
\begin{equation}\label{eq1.3}
K^{2 - p - n} M_p(\Gamma) \le M_p(f(\Gamma)) \le
K^{p + n -2} M_p(\Gamma).
\end{equation}
for every
$p \in ]1,\infty[$
and any family
$\Gamma$
of paths
$\gamma$
such that
$\Im \gamma \subset \cl U_1$.
\end{theorem}

\textbf{Remark~1.1.}
The quantity
$M_p(\Gamma)$,
where
$1 \le p < \infty$,
is called the
$p$-modulus
of the path family
$\Gamma$
and defined by analogy with the conformal modulus
$M_n(\Gamma)$
as
\begin{equation}\label{eq1.4}
M_p(\Gamma) = \inf_{\rho \in \mathcal R(\Gamma)}
\int_{\mathbb R^n} [\rho(x)]^p dx,
\end{equation}
where
$\mathcal R(\Gamma)$
is the set of all nonnegative Borel measurable functions
$\rho: \mathbb R^n \to \dot{\mathbb R}$
such that $\int_{\gamma} \rho ds \ge 1$
for every rectifiable path
$\gamma \in \Gamma$.

It is also well known that if the boundaries of domains
$U_1$
and
$U_2$
are sufficiently regular (e.g., these domains belong to the
class
$\mathcal L$),
then any
$K$-quasi-isometric
homeomorphism of these domains admits a natural extension to a
$K$-quasi-isometric
homeomorphism
$H$
of their closures
$\cl U_1$
and
$\cl U_2$
satisfying condition~(\ref{eq1.3}). In particular,~(\ref{eq1.3})
holds for the
$p$-moduli
$M^{U_1}_p(F) = M^{U_1}_p(\{F_1,F_2\}) = M_p(\Gamma_{F_1,F_2,U_1})$
of the boundary condensers
$F = \{F_1,F_2\}$
of the domain
$U_1$,
and if
$K = 1$
then inequalities~(\ref{eq1.3}) turn into the equality
$$
M^{U_2}_p(f(F)) = M^{U_1}_p(F)
$$
(in this case, the mapping
$f$
in~(\ref{eq1.3}) coincides with the restriction
$H|_{\fr U}$
of~$H$
to the boundary of
$U$).

These facts and Theorems~\ref{t2.1} and~\ref{t2.2} lead to the
following question: Do there exist analogs of
Theorems~\ref{t2.1} and~\ref{t2.2} characterizing the boundary
values of isometric mappings in terms of the
$p$-moduli
of path families? In Sec.~3, we answer this question in the positive.

In the Appendix, for the reader's convenience, we expose the proof
of Theorem~\ref{t1.2}

In what follows, for
$x \in \mathbb R^n$
and
$E \subset \mathbb R^n$,
$\dist(x,E) = \inf\limits_{y \in E} |x - y|$, all paths
$\gamma: [\alpha,\beta] \to \mathbb R^n$,
where
$\alpha,\beta \in \mathbb R$,
are assumed continuous and non-constant, and
$l(\gamma)$
means the length of a path
$\gamma$.

\section{Unique Determination of Nonconvex Polyhedral Domains by
the Relative Conformal Moduli of Their Boundary Condensers}\label{s2}

Let
$\mathcal P_1 = \mathcal P_1(n)$
be the class of all bounded domains
$U$
in
$\mathbb R^n$
satisfying the following conditions:

$(i)$
$\fr U$
is an
$(n - 1)$-manifold
of class
$C^0$
without boundary;

$(ii)$
$\fr U$
can be represented as a finite union of pairwise nonoverlapping
$(n - 1)$-dimensional
cells.

\textbf{Remark~2.1.}
Recall (see, e.g.,~\cite{W}) that a cell
$\sigma$
in
$\mathbb R^n$
is a nonempty closed bounded subset in~$\mathbb R^n$
which can be represented as a finite intersection of closed
half-spaces. The plane
$\Omega(\sigma)$
of a cell
$\sigma$
is the minimal affine subspace containing
$\sigma$;
the dimension
$\dim(\sigma)$
of a cell
$\sigma$
coincides with that of the plane
$\Omega(\sigma)$,
and if this dimension equals~$r$ then
$\sigma$
is called an
$r$-cell.
If the dimensions of two cells
$\sigma_1$
and
$\sigma_2$
coincide and
$\{\inter(\sigma_1)\} \cap \{\inter(\sigma_2)\} = \varnothing$
then we say that the cells
$\sigma_1$
and
$\sigma_2$
do not overlap. Here
$\inter(\sigma_j) = \sigma_j \setminus \partial \sigma_j$
is the open kernel of the cell
$\sigma_j$
($j = 1,2$),
and
$\partial \sigma_j$
denotes the boundary of the cell
$\sigma_j$
treated as a subset of the plane
$\Omega(\sigma_j)$.

The first main result of the article is the following

\begin{theorem}\label{t2.1}
Suppose that
$n \ge 4$.
Then every domain~$U$
in~$\mathbb R^n$
belonging to the class
$\mathcal P_1$
and having connected boundary is uniquely determined in this
class by the relative conformal moduli of its boundary
condensers. Moreover{\rm,}
$U$
can be determined in the class
$\mathcal P_1$
up to an additional affine conformal transformation {\rm(}i.e.{\rm,}
a similarity transformation{\rm)}
$P: \mathbb R^n \to \mathbb R^n$.
\end{theorem}

\textbf{Remark~2.2.}
It should be mentioned in relation to Theorem~\ref{t2.1} that
the boundary of any domain of class
$\mathcal P_1$
is a Lipschitz
($n - 1$)-manifold
without boundary (this follows directly from the definition of
the class
$\mathcal P_1$).

To prove Theorem~\ref{t2.1}, we first need to introduce a number
of notions and remind some assertions of auxiliary nature
from author's article~\cite{Ko2}.

We will begin with the notion of an
($n - 1$)-face
of the boundary of a domain
$U$
from the class
$\mathcal P_1$.

Let
$\Xi$
be a collection of pairwise nonoverlapping cells of dimension
$n - 1$
whose union coincides with the boundary of the domain
$U$,
and suppose that a hyperplane
$\tau$
contains at least one cell from
$\Xi$.
We say that
$\delta$
is an
($n - 1$)-face
of the boundary
$\fr U$
of
$U$
contained in
$\tau$
if it is a maximal union
\begin{equation}\label{eq3.1}
\delta = \bigcup\limits_{s = 1}^k \xi_s
\end{equation}
of those cells in~$\Xi$
that
$(i)$
are contained in~$\tau$,
moreover,
$(ii)$
the interior
$\inter_{\fr U} \delta$
of this union (calculated in the interior metric of the boundary
$\fr U$
of
$U$)
is a connected set.
The maximality of the union~(\ref{eq3.1}) means that it is
impossible to add (at least) one more cell from
$\Xi$
to this union with the preservation of properties
$(i)$
and
$(ii)$.
Clearly, the just-introduced definition of
($n - 1$)-face
of the boundary
$\fr U$
is correct and
$\fr U$
itself is a (uniquely defined) union of pairwise nonoverlapping
($n - 1$)-faces.
Moreover, since the boundary of
$U$
is an
($n - 1$)-dimensional
manifold of class
$C^0$
without boundary, we have the following assertions:

\begin{lemma}\label{l2.1}
If
$\delta$
is
an $(n - 1)$-face
of the boundary of
$U$
contained in a hyperplane
$\tau$
then the set
$\inter \delta\,\, (= \inter_{\fr U} \delta)$
has an open neighborhood
$D$
such that
$D \cap \mathbb R^n_1 \subset U$
and
$D \cap \mathbb R^n_2 \subset \mathbb R^n \setminus U,$
where
$\mathbb R^n_1$
and
$\mathbb R^n_2$
are open half-spaces satisfying the conditions
$\mathbb R^n_1 \cap \mathbb R^n_2 = \varnothing$
and
$\mathbb R^n_1 \cup \mathbb R^n_2 = \mathbb R^n \setminus \tau$.
\end{lemma}

\begin{lemma}\label{l2.2}
Let
$\{\delta_j\}_{j = 1,\dots,k}$
be a proper subset of the set of all
$(n - 1)$-faces
of the boundary
$\fr U$.
Then the boundary
$\varkappa = \fr_{\fr U} \bigl(\bigcup\limits_{j = 1}^l \delta_j\bigr)$
of the union
$\bigcup\limits_{j = 1}^l \delta_j$
with respect to
$\fr U$
is a nonempty subset of the union of boundaries of the cells
$\xi \in \Xi$
and so
$\varkappa$
is the union of a finite set
$\Theta = \Theta(\varkappa)$
of pairwise nonoverlapping
$(n - 2)$-dimensional
cells
$v$.

Furthermore, if $x \in \inter v$ $(v \in \Theta)$
then the contingency
$\contg_U x$
of~$U$
at~$x$
is the set
$\widetilde{V}_{\alpha} = \cl(P(V_{\alpha})),$
where
$P$
is a similarity transformation and
\begin{multline}\label{eq3.2}
V_{\alpha} = \{x = (x_1,x_2,\dots,x_{n - 2},x_{n - 1},x_n) \in
\mathbb R^n: x_j \in \mathbb R, j = 1,2,\dots,n - 2,
\\
x_{n - 1} = r\cos\theta, x_n = r\sin\theta, 0 < r < \infty,
0 < \theta < \alpha\},
\end{multline}
$0 < \alpha < 2\pi,$
$\alpha \ne \pi,$
moreover{\rm,} there exists a number
$r = r_x > 0$
such that
$B(x,r) \cap U = B(x,r) \cap \widetilde{V}_{\alpha}$.
\end{lemma}

\textbf{Remark~2.3.}
The proofs of Lemmas~\ref{l2.1} and~\ref{l2.2} are rather simple. For this reason,
we omit them.

\begin{lemma}\label{l2.3}
Let
$p_1,p_2$
be points of the hyperplane
$\tau_{n - 1} = \{x \in \mathbb R^n: x_n = 0\}$
with
$|p_j| = 1$
$(j = 1,2)$
and let
$F_1,F_2$
be disjoint continua on
$\tau_{n - 1}$
such that
$F_1$
is bounded and contains the points
$0$
and
$p_1,$
whereas
$F_2$
is unbounded and contains
$p_2$.
Then
\begin{equation}\label{eq3.3}
M^{\mathbb R^n_+}(\{F_1,F_2\}) \ge \lambda_n =
\frac{(n - 1)v_{n - 1}\log 3}{16} \Biggl(\frac{[\Gamma(\frac{1}{2(n - 1)})]^2}
{\Gamma(\frac{1}{n - 1})}\Biggr)^{1 - n},
\end{equation}
where
$\mathbb R^n_+ = \{x \in \mathbb R^n: x_n > 0\},$
$v_n$
is the volume of the
$n$-dimensional
unit ball, and
$\Gamma$
is the Euler gamma-function.
\end{lemma}

\textbf{Remark~2.4.}
Lemma~\ref{l2.3} goes back to Theorem~3.10 in~\cite{V1} and the
results of~\cite{Ge2} and~\cite{Ge3}. By Theorem~3.10
in~\cite{V1}, for example,
\begin{equation}\label{eq3.4}
M^{\mathbb R^3 \setminus \{F_1 \cup F_2\}} (\{F_1,F_2\}) \ge \lambda_3,
\end{equation}
where
$F_1$
and
$F_2$
are disjoint continua in
$\mathbb R^3$
such that
$F_1$
is bounded and contains the points
$0$
and
$p_1$,
and
$F_2$
is unbounded and contains
$p_2$,
moreover,
$|p_j| = 1$,
$j = 1,2$.
In~\cite{V1}, there was choisen a way, which made it possible to
obtain estimate~(\ref{eq3.4}) by rather rough but direct
calculations. Using the same calculations, the author proved
Lemma~\ref{l2.3} in~\cite{Ko2} (see Lemma~8.1 in~\cite{Ko2}).

\begin{lemma}\label{l2.4}
The relative conformal modulus
$M(A_t) = M^{\mathbb R^n_+} (A_t)$
of the boundary condenser
$A_t$
$(0 < t < \infty)$
of the half-space
$\mathbb R^n_+$
whose components are segments
\begin{equation}\label{eq3.5}
F_1 = F_1(t) = \{x \in \mathbb R^n: |x_1| \le \frac{t}{2},
x_2 = -\frac{1}{2}, x_j = 0, j = 3,4,\dots,n\}
\end{equation}
and
\begin{equation}\label{eq3.6}
F_2 = F_2(t) = \{x \in \mathbb R^n: |x_1| \le \frac{t}{2},
x_2 = \frac{1}{2}, x_j = 0, j = 3,4,\dots,n\},
\end{equation}
has the following properties{\rm:}
{\rm(i)}
$0 < M(A_t) < \infty;$
{\rm(ii)}
$M(A_t) \to 0$
as
$t \to 0$
and
$M(A_t) \to \infty$
as
$t \to \infty;$
{\rm(iii)}
$M(A_t)$
is an increasing function {\rm(}in the wide sense{\rm)} of the parameter~$t:$
$M(A_{t_1}) \le M(A_{t_2})$
if
$t_1 < t_2;$
{\rm(iv)}
if
$0 < t_1 < t_2 < \infty$
then
$t^{-1}_2 M(A_{t_2}) \le t^{-1}_1 M(A_{t_1})$.
\end{lemma}

\textbf{Remark~2.5.}
The condenser
$A_t$
was first considered (in the case
$n = 3$)
in~\cite{GV} in connection with study of the boundary values of
quasiconformal mappings of domains in~$\mathbb R^3$:
by Lemma~3.6 in~\cite{GV}, the function
$$
\zeta_3(t) = t^{-1} M^{\mathbb R^3 \setminus \{F_1(t) \cup F_2(t)\}}
(\{F_1(t),F_2(t)\}), \quad 0 < t < \infty,
$$
decreases (in the wide sense:
$\zeta_3(t_2) \le \zeta_3(t_1)$
if
$0 < t_1 < t_2 < \infty$),
moreover,
$\zeta_3(t) \to \alpha \in ]0,1[$
as
$t \to \infty$.
In contrast to~\cite{GV}, here (as well as
in~\cite{Ko2}), we consider
$A_t$
as a boundary condenser of the half-space
$\mathbb R^n_+$
and use its properties from Lemma~\ref{l2.4}, whose proof can be
found in~\cite{Ko2} (see Lemma~8.2 in~\cite{Ko2}). Below, we use
the notion of a growth point
$t_n$
of the function
$M^{\mathbb R^n_+}(A_t)$
which is introduced in~\cite{Ko2} as follows:
$t_n$
is a point in $]0,\infty[$
such that
\begin{equation}\label{eq3.7}
M^{\mathbb R^n_+}(A_t) > M^{\mathbb R^n_+}(A_{t_n})
\end{equation}
for every
$t > t_n$
(the existence of growth points for
$M^{\mathbb R^n_+}(A_t)$
ensues directly from Lemma~\ref{l2.4}. For definiteness, we will further assume that
$t_n$
is the least number
$t^*_n \ge 1$
satisfying~(\ref{eq3.7}) if inserted instead of
$t_n$.
Clearly, this number is a growth point of
$M^{\mathbb R^n_+}(A_t)$
too.)

\begin{lemma}\label{l2.5}
Assume that
$0 < t < \infty,$
$A_t = \{F_1,F_2\} = \{F_1(t),F_2(t)\}$
is the boundary condenser of the half-space
$\mathbb R^n_+$
whose components are the segments
$F_1 = F_1(t)$
and
$F_2 = F_2(t)$
defined by~{\rm(\ref{eq3.5})} and~{\rm(\ref{eq3.6})} and
$$
F^{\tau}_j = \{x \in \mathbb R^n: \dist(x,F_j) \le \tau,\,\, x_n \ge 0\},
\quad 0 < \tau < 1/2, \quad j = 1,2.
$$
Then
$$
M(A_t) = \lim_{\tau \to 0} M_n(\Gamma(\tau)),
$$
where
$\Gamma(\tau) = \Gamma_{F^{\tau}_1,F^{\tau}_2,\{F^{\tau}_1
\cup F^{\tau}_2\}}$
is the family of paths joining
$F^{\tau}_1$
and
$F^{\tau}_2$
in
$\mathbb R^n_+ \setminus \{F^{\tau}_1 \cup F^{\tau}_2\}$.
\end{lemma}

\textbf{Remark~2.6.}
Lemma~\ref{l2.5} goes back to Lemma~3.4 in~\cite{GV} about the
continuity of moduli, by which
$$
M(\Gamma) = \lim_{\tau \to 0} M_n(\Gamma(\tau)),
$$
where
$\Gamma$
and
$\Gamma(\tau)$
are the families of curves joining disjoint bounded continua
$E_1$
and
$E_2$
and
$E_1(\tau)$
and
$E_2(\tau)$,
respectively in an open set $U \subset \mathbb R^3$
(here
$E_j(\tau) = \{x \in U: \dist(x,E_j) \le \tau\},\,\, j = 1,2$,
moreover,
$\tau$
is sufficiently small). The proof of Lemma~\ref{l2.5} can be
found in~\cite{Ko2} (see Lemma~8.4 in~\cite{Ko2}).

\begin{lemma}\label{l2.6}
Assume that
$n \ge 4$
and
$0 < \alpha \le 2\pi$.
Put
$F_1 = \{x \in \mathbb R^n : -1 \le x_1 \le 0,
x_j = 0, j = 2,3,\dots,n\}$
and
$F_2 = \{x \in \mathbb R^n : 1 \le x_1 \le \infty,
x_j = 0, j = 2,3,\dots,n\}$.
If
$\Gamma_{\alpha}$
is the family of all paths connecting
$F_1$
and
$F_2$
in
$V_{\alpha}$
then
$M^{V_{\alpha}}(A) = M_n(\Gamma_{\alpha}) =
\frac{\alpha}{\pi}M_n(\Gamma_{\pi}) =
\frac{\alpha}{\pi}M^{V_{\pi}}(A),$
where
$A = \{F_1,F_2\}$
is the boundary condenser of~$V_{\alpha}$
with components
$F_1$
and
$F_2;$
moreover{\rm,}
$V_{\alpha} = \{(x_1,x_2,\dots,x_{n - 2},x_{n - 1},x_n) \in
\mathbb R^n : x_j \in \mathbb R, j = 1,2,\dots,n - 2,
x_{n - 1} = r\cos\theta, x_n = r\sin\theta, 0 < r < \infty,
0 < \theta < \alpha\},$
$0 < \alpha \le 2\pi$.
\end{lemma}

\textbf{Remark~2.7.}
Lemma~\ref{l2.6} is a generalization of Lemma~7.1 in~\cite{GV}
concerning the case
$n = 3$
to~$n \ge 4$
(see Lemma~8.7 in~\cite{Ko2}).

\begin{proof}[Proof of Theorem~{\rm\ref{t2.1}}]
Suppose that
$U_1$
is as in the hypothesis of the theorem,
$U_2$
is a domain of class
$\mathcal P_1$,
and
$f : \fr U_1 \to \fr U_2$
is a homeomorphism of the boundary
$\fr U_1$
of~$U_1$
onto the boundary
$\fr U_2$
of~$U_2$
preserving the relative conformal moduli of boundary condensers.
It is sufficient to show that there exists a similarity
transformation
$F : \mathbb R^n \to \mathbb R^n$
satisfying the condition
$U_2 = F(U_1)$
(and
$f = F|_{\fr U_1}$).

To this end, note first that, by the connectedness of the boundary
$\fr U_1$
of~$U_1$,
and the fact that $f$ is a homeomorphism, the boundary
$\fr U_2$
of~$U_2$
is also connected, and then consider an
($n - 1$)-dimensional
face
$s$
of the boundary
$\fr U_1$
of the domain
$U_1$.
Clearly, there exists an
($n - 1$)-dimensional
face
$\widetilde{s}$
of the boundary
$\fr U_2$
of~$U_2$
such that
$$
B(x,r) \cap f(s) = B(x,r) \cap \inter \widetilde{s} \ne\varnothing
$$
for some
$x \in \inter \widetilde{s}$
and
$r > 0$.

Let
$\widetilde{\sigma}$
be a connected component of the set
$\inter \widetilde{s} \cap \inter_{\fr U_2} f(s)$.
Assume that
$\sigma = f^{-1}(\widetilde{\sigma})$.
We assert that the restriction
$f|_{\sigma}$
of
$f$
to
$\sigma$
is an
($n - 1$)-dimensional
conformal mapping.

Indeed, taking into account that the relative conformal modulus
$M^U(F)$
of a boundary condenser
$F$
of~$U$
is a conformal invariant and using Lemma~\ref{l2.1}, we can
(applying additional conformal mappings if necessary) come to
the following situation:
$s$
and
$\widetilde{s}$
are subsets of the hyperplane
$\tau_{n - 1} = \{x \in \mathbb R^n : x_n = 0\}$,
and the sets
$\inter s$
and
$\inter \widetilde{s}$
have open neighborhoods
$D_1$
and
$D_2$
(respectively) such that
$D_j \cap \mathbb R^n_+ \subset U_j$
and
$D_j \cap \mathbb R^n_- \subset \mathbb R^n \setminus U_j$,
where
$j = 1,2$
and
$\mathbb R^n_- = \mathbb R^n \setminus (\cl \mathbb R^n_+)$.
Suppose that
$x_0 \in \sigma$,
$k \in \{2,3,\dots\}$,
and
$r$
is a sufficiently small positive number and consider an
($n - 1$)-dimensional ball
$B_{n - 1}(x_0,r) = \{x \in \mathbb R^n : |x - x_0| < r,
x_n = 0\}$
in~$\tau_{n-1}$ such that
$\cl B_{n - 1}(x_0,r) \subset \sigma$,
$\cl B_{n - 1}(f(x_0),L) \subset \widetilde{\sigma}$
($L = L(x_0,f,r) = \max\limits_{|x - x_0| = r,x \in \tau_{n - 1}}
|f(x) - f(x_0)|$),
$B^+_n(f(x_0),L) = \{y \in \mathbb R^n_+ : |y - f(x_0)| < L, y_n > 0\}
\subset U_2$,
and
\begin{equation}\label{eq3.8}
B^+_n(x_0,kL^*) = \{x \in \mathbb R^n : |x - x_0| < kL^*, x_n > 0\}
\subset \{D_1 \cap \mathbb R^n_+\} \quad (\subset U_1),
\end{equation}
where
\begin{equation}\label{eq3.9}
\{B_n(x_0,kL^*) \cap \tau_{n - 1}\} \subset \sigma
\end{equation}
and
$L^* = \max\{|x - x_0| : x \in E_1\}$,
moreover,
$E_1 = \{y \in \mathbb R^n : |y - f(x_0)| = L, y_n = 0\}$
(note that we can obtain~(\ref{eq3.8}) and~(\ref{eq3.9}) by using
the continuity of
$f$
and the smallness of the~values of
$r$).
The following estimate holds for the relative conformal modulus
$M^{U_2}(\{E_1,E_2\})$
of the boundary condenser
$\{E_1,E_2\}$
of~$U_2$, where
$E_2 = \{y \in \mathbb R^n : |y - f(x_0)| = l, y_n = 0\}$
($l = l(x_0,f,r) = \min\limits_{|x - x_0| = r,x \in \tau_{n - 1}}
|f(x) - f(x_0)|$):
\begin{equation}\label{eq3.10}
nv_n \biggl(\log\frac{L}{l}\biggr)^{1 - n} \ge M^{U_2}(\{E_1,E_2\}).
\end{equation}
This stems from the fact that the family
$\Gamma_{S_L,S_l,A}$
of all paths connecting the spheres
$S_L = \{y \in \mathbb R^n : |y - f(x_0)| = L\}$
and
$S_l = \{y \in \mathbb R^n : |y - f(x_0)| = l\}$
in the spherical ring
$A = \{y \in \mathbb R^n : l < |y - f(x_0)| < L\}$
minorizes the family
$\Gamma_{E_1,E_2,U_2}$
(i.e., for each path
$\gamma : [\alpha,\beta] \to \mathbb R^n$,
$\gamma \in \Gamma_{E_1,E_2,U_2}$,
there exists a segment
$[\varkappa,\delta]$
($\subset [\alpha,\beta]$)
such that
$\gamma|_{[\varkappa,\delta]} \in \Gamma_{S_L,S_l,A}$),
and from assertions~6.4 and~7.5 in~\cite{V}.

Now, estimate
$M^{U_1}(\{f^{-1}(E_1),f^{-1}(E_2)\})$
(which is equal to
$M^{U_2}(\{E_1,E_2\})$
since
$f$
preserves the relative conformal moduli of boundary condensers)
from below. To this end, note first that
$$
M^{U_1}(\{f^{-1}(E_1),f^{-1}(E_2)\}) \ge M_n(\Gamma_k),
$$
where
$\Gamma_k$
is the subfamily of the family
$\Gamma_{f^{-1}(E_1),f^{-1}(E_2),U_1}$
of all paths connecting the components of the boundary condenser
$\{f^{-1}(E_1),f^{-1}(E_2)\}$
of the domain
$U_1$
in this domain, which consists of the paths
$\gamma \in \Gamma_{f^{-1}(E_1),f^{-1}(E_2),U_1}$
such that
$\Im \gamma \subset \cl B^+_n(x_0,kL^*)$.
Furthermore, consider the family
$\Gamma_{f^{-1}(E_1),f^{-1}(E_2),\mathbb R^n_+}$
of paths connecting the components of the condenser
$\{f^{-1}(E_1),f^{-1}(E_2)\}$
in
$\mathbb R^n_+$,
which is also a boundary condenser for the half-space
$\mathbb R^n_+$.
It is clear that
\begin{equation}\label{eq3.11}
\Gamma_{f^{-1}(E_1),f^{-1}(E_2),\mathbb R^n_+} \subset \{\Gamma_k
\cup \Gamma^k\},
\end{equation}
where
$\Gamma^k$ is the subfamily of paths
$\gamma$
in
$\Gamma_{f^{-1}(E_1),f^{-1}(E_2),\mathbb R^n_+}$
satisfying the condition
$\{\Im \gamma \cap (\mathbb R^n \setminus B^+_n(x_0,kL^*))\}
\ne\varnothing$.
Moreover, since
$\Gamma^k$
is minorized by the family
$\Gamma_{S(x_0,kL^*),S(x_0,L^*),A^*}$
of all paths connecting the boundary spheres
$S(x_0,kL^*)$
and
$S(x_0,L^*)$
of the spherical ring
$A^* = \{x \in \mathbb R^n : L^* < |x - x_0| < kL^*\}$
in this ring, by the above-mentioned assertions~6.4 and~7.5
in~\cite{V}, we have
\begin{equation}\label{eq3.12}
M_n(\Gamma^k) \le nv_n \biggl(\log\frac{kL^*}{L^*}\biggr)^{1 - n}
= nv_n (\log k)^{1 - n} = \mu_k \to 0
\end{equation}
as
$k \to \infty$.
On the other hand,~(\ref{eq3.11}) and~(\ref{eq3.12}) imply
$$
M^{\mathbb R^n_+}(\{f^{-1}(E_1),f^{-1}(E_2)\}) \le
M_n(\Gamma_k) + M_n(\Gamma^k) \le M_n(\Gamma_k) + \mu_k.
$$
From these relations and Lemma~\ref{l2.3} it follows that
\begin{equation}\label{eq3.13}
M_n(\Gamma_k) \ge M^{\mathbb R^n_+}(\{f^{-1}(E_1),f^{-1}(E_2)\})
- \mu_k \ge \lambda_n - \mu_k,
\end{equation}
where
$\lambda_n$
is from~(\ref{eq3.3}). Involving also the fact that
$\lambda_n - \mu_k > 0$
when
$k$
is sufficiently large, and reckoning with~(\ref{eq3.10})
and~(\ref{eq3.13}), we easily obtain the relation
\begin{equation}\label{eq3.14}
\frac{L}{l} = \frac{L(x_0,f,r)}{l(x_0,f,r)} \le
\exp\biggl\{\biggl(\frac{nv_n}{\lambda_n -
\mu_k}\biggr)^{\frac{1}{n - 1}}\biggr\}.
\end{equation}
Passing to the limit first as
$r \to 0$
and then as
$k \to \infty$
in~(\ref{eq3.14}), we get
\begin{multline}\label{eq3.15}
H(x,f) = \limsup_{r \to 0}\frac{L(x,f,r)}{l(x,f,r)} \le \Lambda_n =
\\
\exp\biggl\{\biggl(\frac{8nv_n}{(n - 1)v_{n - 1}\log 3}
\biggr)^{\frac{1}{n - 1}} \frac{[\Gamma(\frac{1}{2(n - 1)})]^2}
{\Gamma(\frac{1}{n - 1})}\biggr\}
\end{multline}
for
$x \in \sigma$.
From~(\ref{eq3.15}) we see that
$f|_{\sigma}$
is an
($n - 1$)-dimensional
quasiconformal mapping (moreover, for the same reasons, the
inverse mapping
$(f|_{\sigma})^{-1}$
is also quasiconformal). Following the proof of Theorem~8.1
in~\cite{Ko2} and making necessary corrections
to it concerned with the specific nature of the general case discussed
in Theorem~\ref{t2.1}, we will now show that
$f|_{\sigma}$
is a conformal mapping.

Indeed, assume that
$x_0 \in \sigma$
is a nondegenerate differentiability point of
$f|_{\sigma}$,
i.e.,
$x_0$
is a point at which
$f$
is differentiable, moreover, the value
$J(x_0,f)$
of its Jacobian at
$x_0$
is nonzero (by the just-proven quasiconformality of
$f|_{\sigma}$,
$\mes_{n - 1}$-almost
all points
$x \in \sigma$
have this property; here
$\mes_{n - 1}$
is the
($n - 1$)-dimensional
Lebesgue measure), and suppose that the differential
$f'(x_0)$
is not a conformal mapping. Consider points
$e_1,e_2 \in \tau_{n - 1}$
such that
$|e_j| = 1$
($j = 1,2$)
and
$|f'(x_0)e_1| = \max\limits_e |f'(x_0)e| > |f'(x_0)e_2| =
\min\limits_e |f'(x_0)e|$,
where the maximum and minimum are calculated over the set of all vectors
$e \in \tau_{n - 1}$
with
$|e| = 1$.
By the conformal invariance of the relative conformal moduli of
boundary condensers, we can assume that
$e_1,e_2,\dots,e_n$
is the canonical basis in
$\mathbb R^n$;
$s,\widetilde{s} \subset \tau_{n - 1}$
(as above,
$s$
and
$\widetilde{s}$
are
($n - 1$)-dimensional
faces of the boundaries
$\fr U_1$
and
$\fr U_2$
of the polyhedrons
$\cl U_1$
and
$\cl U_2$
containing
$\sigma$
and
$\widetilde{\sigma} = f(\sigma)$
respectively);
$x_0 = f(x_0) = 0$;
$B^+_n(0, \sqrt{1 + t^2_n}) (= B_n(0,\sqrt{1 + t^2_n}) \cap
\mathbb R^n_+) \subset D_1 \cap \mathbb R^n_+  \subset U_1$,
$\{B_n(0,\sqrt{1 + t^2_n}) \cap \tau_{n - 1}\} \subset \sigma$;
$B^+_n(0,\sqrt{1 + (\Lambda_nt_n)^2}) \subset D_2 \cap
\mathbb R^n_+ \subset U_2$,
$\{B_n(0,\sqrt{1 + (\Lambda_nt_n)^2}) \cap \tau_{n - 1}\} \subset
\widetilde{\sigma}$,
$f'(0)e_2 = e_2$,
and
$f'(0)e_1 = ue_1$
($1 < u \le \Lambda_n$).
Here
$\Lambda_n$
is defined by~(\ref{eq3.15}) and
$t_n$
is a~growth point (see Remark~2.5) of the function
$t \mapsto M^{\mathbb R^n_+}(A_t)$,
$0 < t < \infty$,
where
$A_t$
is the boundary condenser with components~(\ref{eq3.5})
and~(\ref{eq3.6}).

Starting from this situation, consider the
parameter
$\mu = 2,3,\dots$
and the boundary condenser
\begin{multline*}
\mu^{-1}A_{t_n} = \{F^{\mu}_1,F^{\mu}_2\} =
\\
\{\mu^{-1}F_1,\mu^{-1}F_2\} = \{\{x \in \mathbb R^n : \mu x \in F_1\},
\{x \in \mathbb R^n : \mu x \in F_2\}\}
\end{multline*}
of the half-space $\mathbb R^n_+$,
where
$F_j$
($j = 1,2$)
are the components of the boundary condenser
$A_{t_n}$
of $\mathbb R^n_+$ defined by~(\ref{eq3.5}) and~(\ref{eq3.6}) for
$t = t_n$,
and then construct the mapping
$f_{\mu} : \fr U_1 \to \fr U_2$
by setting
$f_{\mu}(x) =\mu f(\mu^{-1}x)$
where
$x \in \fr(\mu U_1)$.
By the nondegenerate differentiability of~$f|_{\sigma}$
(and hence that of the inverse mapping
$(f|_{\sigma})^{-1}$) at~$0$, we have
\begin{equation}\label{eq3.16}
f_{\mu}(x) = Lx + |x|\alpha(\mu^{-1}x), \quad x \in \fr(\mu U_1),
\end{equation}
and
\begin{equation}\label{eq3.17}
f^{-1}_{\mu}(x) (= \mu f^{-1}(\mu^{-1}x)) =
L^{-1}x + |x|\beta(\mu^{-1}x), \quad x \in \fr(\mu U_2).
\end{equation}
Note that, in~(\ref{eq3.16}) and~(\ref{eq3.17}),
$L = f'(0)$
is the derivative (differential) of the mapping
$f$
at the point
$0$,
\begin{equation}\label{eq3.18}
L\biggl(\frac{\tau e_1  \pm e_2}{2}\biggr) =
\frac{u\tau e_1 \pm e_2}{2}, \quad \tau \in \mathbb R,
\end{equation}
and
\begin{equation}\label{eq3.19}
\lim_{x \to 0, x \in \tau_{n - 1}}\alpha(x) = 0, \quad
\lim_{x \to 0, x \in \tau_{n - 1}}\beta(x) = 0;
\end{equation}
moreover, the mappings
$\alpha$
and
$\beta$
are independent of
$\mu$.

Consider the boundary condenser
$f^{-1}_{\mu}(A_{ut_n})$
of
$\mathbb R^n_+$
with the components
$f^{-1}_{\mu}(F_j(ut_n))$
($j = 1,2$),
where
$F_j$
are defined by~(\ref{eq3.5}) and~(\ref{eq3.6}). For this
condenser, we have
\begin{equation}\label{eq3.20}
\Gamma_{f^{-1}_{\mu}(F_1(ut_n)),f^{-1}_{\mu}(F_2(ut_n)),\mu U_1}
\subset \{\Gamma_{f^{-1}_{\mu}(F_1(ut_n)),f^{-1}_{\mu}(F_2(ut_n)),
\mu B^+_n(0,\sqrt{1 + t^2_n})} \cup \Gamma^*_{\mu}\},
\end{equation}
where
$\Gamma^*_{\mu}$
is the subfamily of paths
$\gamma$
in
$\Gamma_{f^{-1}_{\mu}(F_1(ut_n)),f^{-1}_{\mu}(F_2(ut_n)),\mu U_1}$
such that
$\Im \gamma \cap (\mathbb R^n \setminus
\mu B_n(0,\sqrt{1 + t^2_n})) \ne\varnothing$;
moreover, from the fact that $\Gamma^*_{\mu}$
is minorized by the family of all paths connecting the
boundary spheres of the spherical ring
$\{x \in \mathbb R^n : \sqrt{1 + t^2_n} < |x| <
\mu \sqrt{1 + t^2_n}\}$ in this ring
we (by the same arguments as those used to
deduce~(\ref{eq3.12})) obtain the inequality
\begin{equation}\label{eq3.21}
M_n(\Gamma^*_{\mu}) \le \frac{nv_n}{(\log \mu)^{n - 1}}.
\end{equation}

On the other hand, we can easily verify that
\begin{equation}\label{eq3.22}
\Gamma_{f^{-1}_{\mu}(F_1(ut_n)),f^{-1}_{\mu}(F_2(ut_n)),
\mu B^+_n(0,\sqrt{1 + t^2_n})} \subset
\Gamma_{f^{-1}_{\mu}(F_1(ut_n)),f^{-1}_{\mu}(F_2(ut_n)),
\mathbb R^n_+}.
\end{equation}
Therefore (by~(\ref{eq3.20})-(\ref{eq3.22})),
\begin{equation}\label{eq3.23}
M^{\mu U_1}(f^{-1}_{\mu}(A_{ut_n})) \le
M^{\mathbb R^n_+}(f^{-1}_{\mu}(A_{ut_n})) +
\frac{nv_n}{(\log \mu)^{n - 1}}.
\end{equation}

Now,
\begin{equation}\label{eq3.24}
\Gamma_{F_1(ut_n),F_2(ut_n),\mathbb R^n_+} \subset
\{\Gamma_{F_1(ut_n),F_2(ut_n),\mu B^+_n(0,\sqrt{1 +
(\Lambda_nt_n)^2})} \cup \bar{\Gamma}_{\mu}\},
\end{equation}
where
$\bar{\Gamma}_{\mu}$
is the subset of paths
$\gamma \in \Gamma_{F_1(ut_n),F_2(ut_n),\mathbb R^n_+}$
satisfying the condition
$\Im \gamma \cap \{\mathbb R^n \setminus
(\mu B_n(0,\sqrt{1 + (\Lambda_nt_n)^2}))\} \ne\varnothing$.
Since
$\bar{\Gamma}_{\mu}$
is minorized by the family of all paths connecting the boundary
spheres of the spherical ring
$\{x \in \mathbb R^n : \sqrt{1 + (\Lambda_nt_n)^2} < |x| <
\mu \sqrt{1 + (\Lambda_nt_n)^2}\}$
in this ring, it follows that, by repeating the arguments used
in deriving~(\ref{eq3.12}) and~(\ref{eq3.21}), we get the estimate
\begin{equation}\label{eq3.25}
M_n(\bar{\Gamma}_{\mu}) \le \frac{nv_n}{(\log \mu)^{n - 1}}.
\end{equation}

Using the obvious relation
$$
\Gamma_{F_1(ut_n),F_2(ut_n),\mu B^+_n(0,\sqrt{1 +
(\Lambda_nt_n)^2)}} \subset \Gamma_{F_1(ut_n),F_2(ut_n),\mu U_2},
$$
and~(\ref{eq3.24}) and~(\ref{eq3.25}), we have
$$M^{\mathbb R^n_+}(A_{ut_n}) \le M^{\mu U_2}(A_{ut_n}) +
\frac{nv_n}{(\log \mu)^{n - 1}}.
$$
Involving also~(\ref{eq3.23}) and the circumstance that
the mapping
$f_{\mu}$
(together with~$f$)
preserves the relative conformal moduli of boundary
condensers, we have
\begin{equation}\label{eq3.26}
M^{\mathbb R^n_+}(A_{ut_n}) \le
M^{\mu U_1}(f^{-1}_{\mu}(A_{ut_n})) + \frac{nv_n}
{(\log \mu)^{n - 1}} \le M^{\mathbb R^n_+}(f^{-1}_{\mu}(A_{ut_n})) +
\frac{2nv_n}{(\log \mu)^{n - 1}}.
\end{equation}

The proof of the conformality of
$f|_{\sigma}$
is finished by analogy with that of Theorem~8.1 in~\cite{Ko2}.
We will only confine ourselves to its brief exposition. First,
starting from~(\ref{eq3.16})-(\ref{eq3.19}), we arrive at the
estimate
\begin{equation}\label{eq3.27}
M^{\mathbb R^n_+}(f^{-1}(A_{ut_n})) \le M_n(\Gamma(\beta^*_{\mu})),
\end{equation}
where
\begin{equation}\label{eq3.28}
\beta^*_{\mu} = \frac{\sqrt{1 + (\Lambda_nt_n)^2}}{2}
\biggl\{\sup_{|y| \le \frac{\sqrt{1 + (\Lambda_nt_n)^2}}{2\mu}}
|\beta(y)|\biggr\} \to 0
\end{equation}
as
$\mu \to \infty$
(in what follows,
$\mu$
is so large that
$\beta^*_{\mu} < 1/2$)
and
$\Gamma(\tau) = \Gamma_{F^{\tau}_1,F^{\tau}_2,\mathbb R^n_+
\setminus \{F^{\tau}_1 \cup F^{\tau}_2\}}$
is the family of all paths connecting
$F^{\tau}_1$
and
$F^{\tau}_2$
in
$\mathbb R^n_+ \setminus \{F^{\tau}_1 \cup F^{\tau}_2\}$.
Here
$$
F^{\tau}_j = \{x \in \mathbb R^n : \dist(x,F_j) \le
\tau, x_n \ge 0\},
$$
$0 < \tau < 1/2$
($j = 1,2$),
moreover, the sets
$F_j$
are the components of the boundary condenser
$A(t)$
of
$\mathbb R^n_+$
defined by~(\ref{eq3.5}) and~(\ref{eq3.6}).

Finally, combining~(\ref{eq3.26}) and~(\ref{eq3.27}), we arrive
at the inequalities
\begin{equation}\label{eq3.29}
M^{\mathbb R^n_+}(A_{ut_n}) - \frac{2nv_n}{(\log \mu)^{n - 1}} \le
M_n(\Gamma(\beta^*_{\mu})),
\end{equation}
and then, letting
$\mu$
tend to
$\infty$,
apply Lemma~\ref{l2.5} to the right-hand side in~(\ref{eq3.29}).
In particular, this lemma implies the equality
$$
M^{\mathbb R^n_+}(A_{ut_n}) = \lim_{\tau \to 0} M_n(\Gamma(\tau)).
$$
Now, involving~(\ref{eq3.28}), we obtain the inequality
$$
M^{\mathbb R^n_+}(A_{ut_n}) \le M^{\mathbb R^n_+}(A_{t_n})
$$
which contradicts the fact that
$t_n$
is a~growth point of the function
$t \mapsto M^{\mathbb R^n_+}(A_t)$,
$0 < t < \infty$.
Therefore,
$f|_{\sigma}$
is a conformal mapping.

Note also that since
$\sigma$
is a connected component of the set
$(\inter_{\fr U_1}f^{-1}(\widetilde{s})) \cap \inter s$,
$\widetilde{\sigma} = f(\sigma)$
and
$f^{-1}$
preserves the relative conformal moduli of boundary
condensers, $f^{-1}|_{\sigma}$ is also conformal.

The next step in proving the theorem is the proof of the
equality
$\cl \sigma = s$.

To this end, assume that
$s \setminus \cl \sigma \ne\varnothing$
and then consider a point
$x_0 \in \{(\fr_s \sigma) \cap (\inter s)\}$.
The image
$y_0 = f(x_0)$
of this point belongs to the set
$\partial \widetilde{s}$,
moreover, by the continuity of~$f^{-1}$,
there exists an
$n$-dimensional
ball
$B(y_0,r)$
such that
$f^{-1}(B(y_0,r) \cap \partial \widetilde{s}) \subset
\{(\fr_s \sigma) \cap (\inter s)\}$.
In the set
$B(y_0,r) \cap \partial \widetilde{s}$,
there is a point
$y^*_0$
belonging to the interior
$\inter v$
of a certain
($n - 2$)-dimensional
cell
$v \in \Theta_2$,
where
$\Theta_2$
is the (chosen and fixed a~priori) finite set of pairwise
nonoverlapping
($n - 2$)-dimensional
cells whose union is the boundary
$\varkappa = \fr_{\fr U_2} \widetilde{s}$
of the face
$\widetilde{s}$
(see Lemma~\ref{l2.2}).
Let
$x^*_0 = f^{-1}(y^*_0)$.
Basing on the conformal invariance of the relative conformal
moduli of boundary condensers, assume that
$x^*_0$
is just the initially-chosen point
$x_0$ and, what is more,
$x_0 = f(x_0) = 0$;
$s,\widetilde{s} \subset \tau_{n - 1}$;
the set
$\inter \widetilde{s}$
has an open neighborhood
$D_2$
such that
$(D_2 \cap \mathbb R^n_+) \subset U_2$
and
$(D_2 \cap \mathbb R^n_-) \subset \mathbb R^n \setminus U_2$,
and
$v \subset \{x \in \mathbb R^n : x_{n - 1} = x_n = 0\}$.
Moreover, since the condition
$n \ge 4$
and the well-known properties of space conformal mappings imply
that the~mapping
$f|_{\sigma}$
is a restriction to
$\sigma$
of a certain M\"obius mapping
$h : \bar{\mathbb R}^n \to \bar{\mathbb R}^n$,
we can also assume that
$\widetilde{\sigma} = \sigma$
and
$f|_{\sigma} = \Id \sigma$.

Further, suppose that a number
$r_0 > 0$
satisfies the condition
$\{B(0,r_0) \cap V_{\pi}\} \subset \{D_1 \cap \mathbb R^n_+\}
\subset U_1$
and condition~(\ref{eq3.10}) where now
$x_0 = 0$
and
$r = r_0$,
and let
$\contg_{U_2} 0 = V_{\alpha_2}$
($\alpha_2 \in (]0,\pi[ \cup ]\pi,2\pi[)$);
moreover,
$V_{\alpha}$
is the domain defined for
$\alpha \in ]0,2\pi[$
by~(\ref{eq3.2}).
Setting
$r_0 = 2$
(which is possible because of the conformal invariance of the
relative conformal moduli of boundary condensers), construct
the sequence
$\{f_{\mu}\}_{\mu = 2,3,\dots}$
of the mappings
$f_{\mu} : \fr (\mu U_1) \to \fr (\mu U_2)$,
where (as above)
$f_{\mu}(x) = \mu f(\mu^{-1}x)$,
$x \in \fr (\mu U_1)$.
The mapping
$f_{\mu}$
has the following properties:
\begin{equation}\label{eq3.30}
\{B(0,2\mu) \cap v\} \subset \mu v,
\end{equation}
\begin{equation}\label{eq3.31}
\{B(0,2\mu) \cap V_{\pi}\} \subset \mu U_1,
\end{equation}
\begin{equation}\label{eq3.32}
\{B(0,2\mu) \cap V_{\alpha_2}\} \subset \mu U_2
\end{equation}
and
\begin{equation}\label{eq3.33}
f_{\mu}|_{\mu \sigma} = \Id (\mu \sigma).
\end{equation}

Starting from the mapping
$f_{\mu}$
and taking into account~(\ref{eq3.30})-(\ref{eq3.33}), for the
boundary condenser
$A_{\mu}$
of the domains
$\mu U_j$
($j = 1,2$)
whose components are the sets
$$
F^{\mu}_1 = \{x \in \mathbb R^n : -1 \le x_1 \le 0, x_{\nu} = 0,
\nu = 2,3,\dots,n\}
$$
and
$$
F^{\mu}_2 = \{x \in \mathbb R^n : 1 \le x_1 \le \mu, x_{\nu} = 0,
\nu = 2,3,\dots,n\},
$$
we now obtain the relations
\begin{equation}\label{eq3.34}
M^{V_{\pi}}(A) \le M^{\mu U_1}(A_{\mu}) + M_n(\Gamma_{\mu}) +
M_n(A^*_{\mu},\mathbb R^n)
\end{equation}
$$
(M^{V_{\alpha_2}}(A) \le M^{\mu U_2}(A_{\mu}) + M_n(\Gamma_{\mu})
+ M_n(A^*_{\mu},\mathbb R^n)),
$$
where
$A$
is the boundary condenser of the domains
$V_{\pi}$
and
$V_{\alpha_2}$
(defined above by~(\ref{eq3.2})) with the components
\begin{equation}\label{eq3.35}
F_1 = \{x \in \mathbb R^n : -1 \le x_1 \le 0, x_{\nu} = 0,
\nu = 2,3,\dots,n\}
\end{equation}
and
\begin{equation}\label{eq3.36}
F_2 = \{x \in \mathbb R^n : 1 \le x_1 \le \infty, x_{\nu} = 0,
\nu = 2,3,\dots,n\},
\end{equation}
$\Gamma_{\mu}$
is the family of paths
$\gamma$
connecting of
$F_1$
and
$F_2$
in
$\mathbb R^n \setminus \{F_1 \cup F_2\}$
and such that
$\Im \gamma \cap \{\mathbb R^n \setminus B(0,2\mu)\} \ne\varnothing$;
finally,
$A^*_{\mu}$
is the condenser in
$\mathbb R^n$
whose components are the sets
$F^{\mu}_1$
and
$$
F^{\mu *}_2 = \{x \in \mathbb R^n : \mu \le x_1 < \infty, x_{\nu} = 0,
\nu = 2,3,\dots,n\},
$$
moreover,
$$
M_n(A^*_{\mu},\mathbb R^n) = M_n(\Gamma_{F^{\mu}_1,F^{\mu *}_2,
\mathbb R^n \setminus \{F^{\mu}_1 \cup F^{\mu *}_2\}}).
$$

Indeed,~(\ref{eq3.30})-(\ref{eq3.33}) imply the relations
$$
\Gamma_{F_1,F_2,V_{\pi}} \subset \Gamma_{F^{\mu}_1,F^{\mu}_2,
V_{\pi}} \cup \Gamma_{F^{\mu}_1,F^{\mu *}_2,\mathbb R^n \setminus
\{F^{\mu}_1 \cup F^{\mu *}_2\}}
$$
$$
(\Gamma_{F_1,F_2,V_{\alpha_2}} \subset \Gamma_{F^{\mu}_1,
F^{\mu}_2,V_{\alpha_2}} \cup \Gamma_{F^{\mu}_1,F^{\mu *}_2,
\mathbb R^n \setminus \{F^{\mu}_1 \cup F^{\mu *}_2\}})
$$
and
$$\Gamma_{F^{\mu}_1,F^{\mu}_2,V_{\pi}} \subset \Gamma_{F^{\mu}_1,
F^{\mu}_2,\mu U_1} \cup \Gamma_{\mu}
$$
$$
(\Gamma_{F^{\mu}_1,F^{\mu}_2,V_{\alpha_2}} \subset
\Gamma_{F^{\mu}_1,F^{\mu}_2,\mu U_2} \cup \Gamma_{\mu}).
$$
Thus, by Theorem~6.2 in~\cite{V}, we have
$$
M^{V_{\pi}}(A) \le M^{V_{\pi}}(A_{\mu}) +
M_n(A^*_{\mu},\mathbb R^n) \le M^{\mu U_1}(A_{\mu}) +
M_n(\Gamma_{\mu}) + M_n(A^*_{\mu},\mathbb R^n)
$$
$$
(M^{V_{\alpha_2}}(A) \le M^{V_{\alpha_2}}(A_{\mu}) +
M_n(A^*_{\mu},\mathbb R^n) \le M^{\mu U_2}(A_{\mu}) +
M_n(\Gamma_{\mu}) + M_n(A^*_{\mu},\mathbb R^n)).
$$

Taking into account that the families
$\Gamma_{F^{\mu}_1,F^{\mu *}_2,\mathbb R^n \setminus \{F^{\mu}_1
\cup F^{\mu *}_2\}}$
and
$\Gamma_{\mu}$
are minorized by the families
$\Gamma_{S_1,S'_2,A'_{\mu}}$
and
$\Gamma_{S_1,S''_2,A''_{\mu}}$
of paths connecting the boundary spheres
$S_1 = S_{n - 1}(0,1)$
and
$S'_2 = S_{n - 1}(0,\mu)$
in the spherical ring
$A'_{\mu} = \{x \in \mathbb R^n : 1 < |x| < \mu\}$
and the boundary spheres
$S_1$
and
$S''_2 = S_{n - 1}(0,2\mu)$
in the spherical ring
$A''_{\mu} = \{x \in \mathbb R^n : 1 < |x| < 2\mu\}$,
respectively, by Theorems~6.2,~6.4 and~7.5 in~\cite{V}, we
obtain
\begin{equation}\label{eq3.37}
M_n(A^*_{\mu},\mathbb R^n) \le nv_n (\log \mu)^{1 - n}
\end{equation}
and
\begin{equation}\label{eq3.38}
M_n(\Gamma_{\mu}) \le nv_n \{\log (2\mu)\}^{1 - n} <
nv_n(\log \mu)^{1 - n}.
\end{equation}
Inequalities~(\ref{eq3.34}),~(\ref{eq3.37}) and~(\ref{eq3.38})
imply the relations
\begin{equation}\label{eq3.39}
M^{V_{\pi}}(A) \le M^{\mu U_1}(A_{\mu}) + 2nv_n (\log \mu)^{1 - n}
\end{equation}
$$
(M^{V_{\alpha_2}}(A) \le M^{\mu U_2}(A_{\mu}) +
2nv_n(\log \mu)^{1 - n}).
$$

On the other hand,
$$
\Gamma_{F^{\mu}_1,F^{\mu}_2,\mu U_1} \subset \Gamma_{F^{\mu}_1,
F^{\mu}_2,B(0,2\mu) \cap V_{\pi}} \cup \Gamma^{\mu} \subset
\Gamma_{F^{\mu}_1,F^{\mu}_2,V_{\pi}} \cup \Gamma^{\mu}
$$
$$
(\Gamma_{F^{\mu}_1,F^{\mu}_2,\mu U_2} \subset \Gamma_{F^{\mu}_1,
F^{\mu}_2,B(0,2\mu) \cap V_{\alpha_2}} \cup \Gamma^{\mu} \subset
\Gamma_{F^{\mu}_1,F^{\mu}_2,V_{\alpha_2}} \cup \Gamma^{\mu}),
$$
where
$\Gamma_{F^{\mu}_1,F^{\mu}_2,B(0,2\mu) \cap V_{\pi}}$
($\Gamma_{F^{\mu}_1,F^{\mu}_2,B(0,2\mu) \cap V_{\alpha_2}}$)
is the subfamily of paths in
$\Gamma_{F^{\mu}_1,F^{\mu}_2,\mu U_1}$
($\Gamma_{F^{\mu}_1,F^{\mu}_2,\mu U_2}$)
whose images are in the ball
$B(0,2\mu)$
(note that here we have reckoned with~(\ref{eq3.31})
((\ref{eq3.32}))), and
$\Gamma^{\mu}$
is the subfamily of all paths~$\gamma$
in the same family such that
$\Im \gamma \cap \{\mathbb R^n \setminus B(0,2\mu)\}
\ne\varnothing$;
moreover, just as $\Gamma_{\mu}$, the family $\Gamma^{\mu}$
is minorized by
$\Gamma_{S_1,S''_2,A''_{\mu}}$.
Hence,
\begin{equation}\label{eq3.40}
M^{\mu U_1}(A_{\mu}) \le M^{V_{\pi}}(A_{\mu}) +
nv_n \{\log(2\mu)\}^{1 - n} \le M^{V_{\pi}}(A) +
nv_n (\log \mu)^{1 - n}.
\end{equation}
Combining~(\ref{eq3.39}) and~(\ref{eq3.40}), we finally
prove that
$$
M^{V_{\pi}}(A) - 2nv_n (\log \mu)^{1 - n} \le
M^{\mu U_1}(A_{\mu}) \le M^{V_{\pi}}(A) + nv_n (\log \mu)^{1 - n},
$$
which in turn implies the relation
\begin{equation}\label{eq3.41}
\lim_{\mu \to \infty}M^{\mu U_1}(A_{\mu}) = M^{V_{\pi}}(A).
\end{equation}

Similar arguments also enable us to obtain the inequalities
$$
M^{V_{\alpha_2}}(A) - 2nv_n (\log \mu)^{1 - n} \le
M^{\mu U_2}(A_{\mu}) \le M^{V_{\alpha_2}}(A) +
nv_n (\log \mu)^{1 - n},
$$
which imply the equality
\begin{equation}\label{eq3.42}
\lim_{\mu \to \infty}M^{\mu U_2}(A_{\mu}) = M^{V_{\alpha}}(A).
\end{equation}

Next, the fact that
$f_{\mu}$
(together with
$f$)
preserves the relative conformal moduli of boundary condensers
imply the equality
\begin{equation}\label{eq3.43}
M^{\mu U_1}(A_{\mu}) = M^{\mu U_2}(A_{\mu}).
\end{equation}
Thus, by~(\ref{eq3.41})-(\ref{eq3.43}),
$$
M^{V_{\alpha_2}}(A) = M^{V_{\pi}}(A).
$$
At the same time, Lemma~\ref{l2.6} and the condition
$\alpha_2 \in (]0,\pi[ \cup ]\pi,2\pi[)$
imply the inequality
$$
(0 <)\,\, M^{V_{\alpha_2}}(A) \ne M^{V_{\pi}}(A).
$$

The so-obtained contradiction completes the proof of the equality \linebreak
$\cl \sigma = s$.
It should be noted that, taking $f^{-1}$
instead of
$f$
in the above-mentioned arguments,
we also establish the equality
$\cl \widetilde{\sigma} = \widetilde{s}$.
Hence,
$f$
generates a bijection between the sets of all
($n - 1$)-dimensional
faces of the boundaries
$\fr U_1$
and
$\fr U_2$
of
$U_1$
and
$U_2$.

Turning to the final step in the proof of the theorem, choose an~arbitrary
($n - 1$)-dimensional
face
$s_1$
of the boundary
$\fr U_1$
of the polyhedron
$\cl U_1$.
As above, we may assume that
$s_1 \subset \tau_{n - 1}$,
$f|_{s_1} = \Id s_1$,
$(D_j \cap \mathbb R^n_+) \subset U_j$
and
$(D_j \cap \mathbb R^n_-) \subset \mathbb R^n \setminus U_j$
($j = 1,2$)
($D_j$
are the open neighborhoods of the~faces
$s_1$
and
$\widetilde{s}_1 = f(s_1)$
defined for the domains
$U_j$
by Lemma~\ref{l2.1}). Let
$s_2$
be an
($n - 1$)-dimensional
face of the boundary
$\fr U_1$
such that the intersection
$s_1 \cap s_2$
of
$s_1$
and
$s_2$
contains an
($n - 2$)-dimensional
cell
$v_0$
(from an a~priori fixed finite set
$\Theta$
of pairwise nonoverlapping
($n - 2$)-dimensional
cells whose union is
$\fr_{\fr U_1} s_1$
(see Lemma~\ref{l2.2})). We assert that
$f(s_2) = s_2$.
Indeed, since
$f|_{\inter s_2}$
is a conformal mapping of the
($n - 1$)-dimensional
domain
$\inter s_2$
onto the
($n - 1$)-dimensional
domain
$f(\inter s_2)$,
the condition
$n \ge 4$
and the properties of space conformal mappings imply that
$f|_{s_2} = h|_{s_2}$,
where
$h : \bar{\mathbb R}^n \to \bar{\mathbb R}^n$
is a M\"obius transformation. Taking into account the relation
$f|_{v_0} = \Id v_0$,
we conclude that
$h$
is an isometric mapping of
$\mathbb R^n$.
Let
$x_0 \in \inter v_0$.
Repeating the arguments used above for proving the equality
$\cl \sigma = s$
almost verbatim and applying Lemmas~\ref{l2.2} and~\ref{l2.6},
we obtain the equality
$\contg_{U_1} x_0 = \contg_{U_2} x_0$\,\,
($= \contg_{U_2} f(x_0)$),
from which (and what was said above) we have the desired equality
$f(s_2) = s_2$.

Continuing these arguments by induction, say, at
$l${th}
step, we will either establish the conformal equivalence of the
domains
$U_1$
and
$U_2$
or obtain the following situation: there exists a proper subset
$\{s_{\nu} : \nu = 1,2,\dots,l\}$
of the set of all
($n - 1$)-dimensional
faces of the boundary
$\fr U_1$
such that
$$
f\bigl|_{\bigcup\limits_{\nu = 1}^l s_{\nu}} =
(P \circ L)\bigl|_{\bigcup\limits_{\nu = 1}^l s_{\nu}},
$$
where
$P : \mathbb R^n \to \mathbb R^n$
is an isometry and
$L$
is a M\"obius transformation. Show that, in the so-obtained
situation, we can make at least one more step.

Indeed, consider the set
$\{s_{\nu} : \nu = l + 1,l + 2,\dots,m\}$
of all remaining
($n - 1$)-dimensional
faces of the boundary
$\fr U_1$.
Applying to this set Lemma~\ref{l2.2} and comparing it with the
set
$\{s_{\nu} : \nu = 1,2,\dots,l\}$,
it is easy to conclude that there are faces
$s_{\nu_1} \in \{s_{\nu} : \nu = 1,2,\dots,l\}$
and
$s_{\nu_2} \in \{s_{\nu} : \nu = l + 1,l + 2,\dots,m\}$
such that their intersection contains
an ($n - 2$)-dimensional
cell
$v_0$.
As a result, for the pair of
($n - 1$)-dimensional
faces
$s_{\nu_1}$
and
$s_{\nu_2}$,
we find ourselves in the situation described above (at the~first
step) for~$s_1$ and~$s_2$.
Therefore, it is not difficult to conclude that
$$
f\bigl|_{\bigcup\limits_{\nu = 1}^{l + 1} s_{\nu}} =
(P \circ L)\bigl|_{\bigcup\limits_{\nu = 1}^{l + 1} s_{\nu}}
$$
where now
$s_{l + 1} = s_{\nu_2}$.
Continuing our arguments by induction and taking into account the
finiteness of the set of all
($n - 1$)-dimensional
faces of the boundary
$\fr U_1$
of the domain
$U_1$,
we finally obtain the conformal equivalence of~$U_1$
and~$U_2$.

The existence of a similarity transformation
$P : \mathbb R^n \to \mathbb R^n$
satisfying the condition
$U_2 = P(U_1)$
can be established in the same way as in the proof of Theorem~8.1
in~\cite{Ko2}. Namely, if
$H : U_1 \to \mathbb R^n$
is a conformal mapping from~$U_1$
into
$\mathbb R^n$
that is not the restriction to
$U_1$
of a similarity transformation then the image
$\widetilde{H}(s)$
of at least one of the
($n - 1$)-dimensional
faces
$s$
of the boundary
$\fr U_1$
under a conformal mapping
$\widetilde{H} : \bar{\mathbb R}^n \to \bar{\mathbb R}^n$
such that
$H = \widetilde{H}|_{U_1}$
is a subset of a certain sphere
$S_{n - 1}(x,r)$,
$x \in \mathbb R^n$,
$r \in \mathbb R_+$.
But this is impossible because
$H(U_1) \in \mathcal P_1$.
Thus, the theorem is completely proved.
\end{proof}

Further, let a domain
$U \subset \mathbb R^n$
($n \ge 3$)
be such that there exist a convex domain
$V \subset \mathbb R^n$
and an at most countable set
$\Lambda = \{\lambda_j\}$
of hyperplanes
$\lambda_j$
satisfying the following conditions: ${\rm(i)}$ the intersection
$s_j = \lambda_j \cap \fr V$
of each hyperplane
$\lambda_j \in \Lambda$
with the boundary
$\fr V$
of the domain
$V$
is an
($n - 1$)-dimensional
convex set;
${\rm(ii)}$
$\fr V = (\bigcup\limits_j s_j) \bigcup (\bigcup\limits_{\nu = 1}^k
\{x_{\nu}\})$,
where the union
$E = \bigcup\limits_{\nu = 1}^k \{x_{\nu}\}$
is finite and consists of singletons
$\{x_{\nu}\}$,
moreover, if~$V$
is bounded then
$x_{\nu} \in \mathbb R^n$
for
$\nu = 1,2,\dots,k$,
and if
$V$
is unbounded then
$x_{\nu} \in \mathbb R^n$
for
$\nu = 1,2,\dots,k - 1$
and
$x_k = \infty$;
finally, for every neighborhood
$W$
of~$E$ in~$\bar{\mathbb R}^n$,
the relation
$\{(\fr V) \setminus W\} \cap s_j \ne\varnothing$
holds for at most finitely many subscripts
$j$;
${\rm(iii)}$
$U = \Phi(V)$,
where
$\Phi : \bar{\mathbb R}^n \to \bar{\mathbb R}^n$
is a homeomorphism with the following properties:
${\rm(\circ)}$
$\Phi(\infty) = \infty$,
${\rm(\circ\circ)}$
$\Phi|_{\mathbb R^n}$
is a bi-Lipschitz mapping, and
${\rm(\circ\circ\circ)}$
for each
$j$,
the restriction
$\Phi|_{s_j}$
coincides with the restriction
$\Phi_j|_{s_j}$
to
$s_j$
of some affine mapping
$\Phi_j : \mathbb R^n \to \mathbb R^n$.

We denote the class of all domains
$U$
of the form described above by
$\mathcal P_2 = \mathcal P_2(n)$.
Theorem~\ref{t2.1} is naturally supplemented by the following
assertion.

\begin{theorem}\label{t2.2}
If
$n \ge 4$
then every domain of class
$\mathcal P_2$
is uniquely determined in this class by the relative conformal
moduli of boundary condensers. Moreover,
$U$
can be determined in
$\mathcal P_2$
up to an additional similarity transformation.
\end{theorem}

Though the structure of the class
$\mathcal P_2$
is similar to that of
$\mathcal P_1$,
it still contains unbounded domains of polyhedral
type. Thus, Theorem~\ref{t2.2} makes it possible to waive not only
the convexity but also the boundedness of domains
in Theorem~\ref{t2.1}.

\textbf{Remark~2.8.}
If the components of a boundary condenser
$F$
of a domain~$U$
are connected then this condenser is a ring (in the sense
of~\cite{V}). It is well known that, in this case, the relative
conformal modulus of the condenser
$F$
is equal to its relative conformal capacity
$\capac^U (F) = \capac^U_n (F)$,
i.e., its
$n$-capacity
with respect to the domain
$U$
(see, e.g.,~\cite{Ge4} for the definition of the
$p$-capacity
of a ring and its very close relationship with the theory of
quasiconformal and quasi-isometric (bi-Lipschitz) mappings). The
proofs of Theorems~\ref{t2.1} and~\ref{t2.2} use only ring-shaped
boundary condensers. This allows us to reformulate
Theorem~\ref{t2.1} as follows:

\newtheorem*{theorem-shtrikh}{Theorem 2.1$'$}

\begin{theorem-shtrikh}
If
$n \ge 4$
then any domain
$U$
in
$\mathbb R^n$
belonging to the class
$\mathcal P_1(n)$
and having connected boundary is uniquely determined in this
class by the relative conformal capacities of its ring-shaped
boundary condensers.
\end{theorem-shtrikh}

Theorem~\ref{t2.2} admits a similar reformulation.
\vskip4mm

%\textit{Proof of Theorem~{\rm\ref{t3.2}}}
\begin{proof}[Proof of Theorem~{\rm\ref{t2.2}}]
The proof follows the lines of the proof of Theorem~\ref{t2.1}.

In this case, a certain peculiarity appears by the fact
that for the set~$E$
of item {\rm(ii)} of the definition of the~convex domain
$V$
which has properties {\rm(i)} and {\rm(ii)} in the definition of
class
$\mathcal P_2$
and is connected with the domain
$U_1$
(here
$U_j$
($j = 1,2$)
are domains in the~class
$\mathcal P_2$
such that there exists a homeomorphism
$f : \fr U_1 \to \fr U_2$
of the boundary
$\fr U_1$
of
$U_1$
onto the boundary
$\fr U_2$
of
$U_2$
preserving the relative conformal moduli of boundary condensers)
by the relation
$U_1 = \Phi(V)$,
where
$\Phi : \bar{\mathbb R}^n \to \bar{\mathbb R}^n$
is a homeomorphism satisfying condition ${\rm(iii)}$ in the
definition of
$\mathcal P_2$,
we first choose a~sequence
$\{W_{\nu}\}_{\nu = 1,2,\dots}$
of neighborhoods
$W_{\nu}$
of~$E$
such that
$W_{\nu + 1} \subset W_{\nu}$,
the set
$\fr V \setminus W_{\nu}$
is connected
($\nu = 1,2,\dots$),
and
$\bigcup\limits_{\nu = 1}^{\infty} W_{\nu} = E$.
Then, acting as in the proof of
Theorem~\ref{t2.1}, we establish that for each
$\nu = 1,2,\dots$,
there exists a similarity transformation
$P_{\nu} : \mathbb R^n \to \mathbb R^n$
such that
$f|_{\Phi(\{\fr V\} \setminus W_{\nu})} =
P_{\nu}|_{\Phi(\{\fr V\} \setminus W_{\nu})}$.
Finally, letting
$\nu$
tend to
$\infty$,
we obtain Theorem~\ref{t2.2}. The details of the argument are left to the~reader.
\end{proof}

\section{Boundary Values of Isometric Mappings and the $p$-Moduli
of Path Families}\label{s3}

The facts of the theory of quasi-isometric mappings stated in the
Sec.~1 and Theorems~\ref{t2.1} and~\ref{t2.2} lead to the
following question: Do there exist analogs of Theorems~\ref{t2.1}
and~\ref{t2.2} characterizing the boundary values of isometric
mappings in terms of
$p$-moduli
of path families? At present, we can give the following answer to
this question:

\begin{corollary}\label{c3.1}{\rm(of Theorems~\ref{t2.1}
and~\ref{t2.2})}
Let
$n \ge 4$.
Suppose that
$U_1$
and
$U_2$
are bounded domains of class
$\mathcal P_1$
$(\mathcal P_2)$
having connected boundaries for which there exist a homeomorphism
$f : \fr U_1 \to \fr U_2$
of the boundaries of these domains and a number
$p \in \{]1,n[ \cup ]n,\infty[\}$
such that the following conditions hold{\rm:}
$f$
preserves both the relative
$n$-moduli
and the relative
$p$-moduli
of boundary condensers. Then there exists an isometry
$H : \mathbb R^n \to \mathbb R^n$
satisfying the condition
$H(U_1) = U_2$.
\end{corollary}

\begin{proof}[Proof]
%\textit{Proof}
The proof of the corollary is based on Theorems~\ref{t2.1}
and~\ref{t2.2} and Lemma~\ref{l3.1} which will be formulated
immediately after the proof of Corollary~\ref{c3.1}.

The hypothesis of the corollary and Theorem~\ref{t2.1}
(Theorem~\ref{t2.2}) imply the existence of an affine conformal
mapping
$H : \mathbb R^n \to \mathbb R^n$
such that
$U_2 = H(U_1)$
and
$H|_{\fr U_1} = f$.
Nevertheless, if
$H$
is not an isometry then it has the form
$H(x) = \varkappa \Omega x + \nu$
($x \in \mathbb R^n$),
where
$0 < \varkappa \ne 1$,
$\nu$
is the fixed point of
$\mathbb R^n$
and
$\Omega$
is an orthogonal mapping. By Lemma~\ref{l3.1}, there exists a
ring-shaped boundary condenser
$F$
of~$U_1$
satisfying the condition
$0 < M^{U_1}_p(F) < \infty$.
Furthermore, Theorem~8.2 in~\cite{V} immediately implies the following
assertion: \textit{if
$\varkappa > 0$
and
$G : \mathbb R^n \to \mathbb R^n$
is an affine conformal mapping{\rm,} i.e.{\rm,} a~mapping
defined by the relation
$G(x) = \varkappa \Omega x + \nu$
{\rm(}$\Omega$
is as above an orthogonal mapping{\rm)} then
$M_p(G(\Gamma)) = \varkappa^{n - p} M_p(\Gamma)$
for any path family
$\Gamma$}.
Using this assertion and considerations from Sec.~1, we get the
relation
\begin{equation}\label{eq4.1}
M^{U_2}_p(f(F)) = M^{U_2}_p(H(F)) = \varkappa^{n - p}M^{U_1}_p(F)
\end{equation}
which is a contradiction to the facts that
$f$
preserves the relative
$p$-moduli
of boundary ring-shaped condensers and
$\varkappa^{n - p} \ne 1$
in~(\ref{eq4.1}) since
$0 < \varkappa \ne 1$.
The corollary is proved.
\end{proof}

\begin{lemma}\label{l3.1}
Assume that a domain
$U$
is bounded and belongs to the class
$\mathcal P_1$
$(\mathcal P_2)$.
Then there exists a ring-shaped boundary condenser
$F$
of~$U$
such that
$0 < M^U_p(F) < \infty$
for every
$p \in [1,\infty[$.
\end{lemma}

\begin{proof}[Proof]
Consider a ball
$B = B(x_0,r)$
satisfying the condition
$\cl B(x_0,r) \subset U$
and then the set
$A = T \cap U$,
where
$T = \{x \in \mathbb R^n :
\sum\limits_{\nu = 1}^{n - 1}(x_{\nu} - x_{0 \nu})^2 < r^2\}$.
Let
$\Gamma$
be the family of all paths
$\gamma : [0,1] \to \cl A$
such that
$$
\gamma(t) = \sum_{\nu = 1}^{n - 1} \alpha_{\nu}(\gamma) e_{\nu} +
\{b_n(\gamma)t + a_n(\gamma)(1 - t)\}e_n,
$$
where
$a_n(\gamma) < b_n(\gamma)$;
$\gamma(0),\gamma(1) \in \fr U$;
$\gamma(t) \in U$
if
$t \in ]0,1[$;
finally,
$B \cap \Im \gamma \ne\varnothing$.
Starting from~$\Gamma$,
construct the boundary condenser
$F = \{F_1,F_2\}$
by setting
$F_1 =
\cl \{\sum\limits_{\nu = 1}^{n - 1} \alpha_{\nu}(\gamma)e_{\nu} +
a_n(\gamma)e_n : \gamma \in \Gamma\}$
and
$F_2 =
\cl \{\sum\limits_{\nu = 1}^{n - 1} \alpha_{\nu}(\gamma)e_{\nu} +
b_n(\gamma)e_n : \gamma \in \Gamma\}$.
Clearly,
\begin{equation}\label{eq4.2}
\inf_{\gamma \in \Gamma_{F_1,F_2,U}} l(\gamma) = \lambda > 0.
\end{equation}
Recalling also that
$U \in \mathcal P_1$
($\in \mathcal P_2$),
it is easy to verify the existence of a cylinder
$T^* = \{x \in \mathbb R^n :
\sum\limits_{\nu = 1}^{n - 1} (x_{\nu} - x_{* \nu})^2 < r^2_*\}$
satisfying the following conditions: {\rm(i)}
$\cl T^* \subset T$
and {\rm(ii)}
$F^*_j = F_j \cap \cl T^*$
is a subset of a certain hyperplane
$\tau_{n - 1,j}$
($j = 1,2$).
We assert that the ring-shaped boundary condenser
$F^* = \{F^*_1,F^*_2\}$
is a desired one.

Indeed, by~(\ref{eq4.2}) and Theorem~7.1 in~\cite{V},
$$
M^U_p(F^*) \le \frac{\mes U}{\lambda^p} < \infty.
$$
On the other hand, the boundedness of~$U$
implies the existence of numbers
$a^*_n,b^*_n$
($0 < a^*_n < b^*_n < \infty$)
having the following properties:
$a^*_n \le a_n(\gamma^*) < b_n(\gamma^*) \le b^*_n$
for every path
$\gamma^* : [0,1] \to \mathbb R^n$
of the form
\begin{equation}\label{eq4.3}
\gamma^*(t) = \sum_{\nu = 1}^{n - 1} \alpha_{\nu} e_{\nu} +
\{b^*_n t + a^*_n (1 - t)\}e_n, \quad t \in [0,1],
\end{equation}
where
$\sum\limits_{\nu = 1}^{n - 1} (\alpha_{\nu} - x_{* \nu})^2 <
r^2_*$.
The family
$\Gamma_{F^*_1,F^*_2,U}$
minorizes the family
$\Gamma^*$
of all paths having the form~(\ref{eq4.3}).
Hence, assertions~6.4 and~7.2 in~\cite{V} imply the relations
$$
M^U_p(F^*) \ge M_p (\Gamma^*) = \frac{r^{n - 1}_* v_{n - 1}}
{(b^*_n - a^*_n)^{p - 1}} > 0.
$$
The lemma is proved.
\end{proof}
\medskip

\section{Appendix}\label{s4}

\begin{proof}[Proof of Theorem~{\rm\ref{t1.2}}]
The proof of this theorem follows the lines of the
proof of the second claim of Theorem~6.5 in~\cite{V1}.
Therefore, we will expose it briefly.

Let
$\Gamma$
be a family of paths in the domain
$U_1$
(i.e., of paths
$\gamma : [a,b] \to \mathbb R^n$
such that
$\Im \gamma \subset U_1$).
Consider the subfamily
$\Gamma^*$
of~$\Gamma$
consisting of all locally rectifiable paths
$\gamma \in \Gamma$
such that
$f$
is absolutely continuous on every closed subpath of
$\gamma$.
Since
$f$
is a quasi-isometry,
$f \in ACL_p$
for all $p > 1$
(see, for example,~\cite{Fu},~\cite{V1} for
the definition of the class $ACL_p$); therefore,
$M_p(\Gamma_0) = 0$
for the family
$\Gamma_0$
of all locally rectifiable paths in
$U_1$
having subpaths on which the mapping
$f$
is not absolutely continuous~(\cite{Fu}). The fact that
$\Gamma \setminus \Gamma^* \subset \Gamma_0$
and the properties of moduli imply the equality
$M_p(\Gamma \setminus \Gamma^*) = 0$.
Consequently,
$M_p(\Gamma^*) = M_p(\Gamma)$.
Therefore, for proving, for example, the left-hand inequality
in~(\ref{eq1.3}), which we will do below, it suffices to show that
$M_p(\Gamma^*) \le K^{p + n - 2}M_p(f(\Gamma))$.

Let
$E$
be a Borel subset in~$U_1$ that contains all points
$x \in U_1$
at which~$f$
is not differentiable and all those points~$x$
in~$U_1$
at which
$f$
is differentiable but the Jacobian
$J(x,f) = 0$,
moreover,
$\mes E\,\, (= \mes_n E) = 0$.
Here we use the facts that a quasi-isometric mapping is
quasiconformal and the set of points of nondegenerate
differentiability of a quasiconformal mapping is a set of full
measure with respect to its domain of definition.

Assume that
$\widetilde{\rho} \in \mathcal R(f(\Gamma^*))$,
i.e.,
$\int_{\widetilde{\gamma}} \widetilde{\rho}(x) ds \ge 1$
for every locally rectifiable path
$\widetilde{\gamma} \in f(\Gamma^*)$.
Define a~function
$\rho : \mathbb R^n \to \mathbb R^n$
by setting
$\rho(x) = \widetilde{\rho}(f(x)) ||f'(x)||$
if
$x \in U_1 \setminus E$,
$\rho(x) = \infty$
if
$x \in E$,
and
$\rho(x) = 0$
if
$x \in \mathbb R^n \setminus U_1$.
Arguing as in the proof of the second part of Theorem~6.5
in~\cite{V1} (or of Theorem~32.3 in~\cite{V}, which is the
$n$-dimensional
variant of the first theorem), we further infer that
$\rho \in \mathcal R(\Gamma^*)$, and hence
\begin{multline}\label{eq4.1}
M_p(\Gamma) = M_p(\Gamma^*) \le \int_{\mathbb R^n} \rho^p dx =
\int_{U_1}[\widetilde{\rho}(f(x))]^p ||f'(x)||^p dx =
\\
\int_{U_1}[\widetilde{\rho}(f(x))]^p \frac{||f'(x)||^p}{|J(x,f)|}
|J(x,f)| dx \le
K^{p + n - 2} \int_{U_1} [\widetilde{\rho}(f(x))]^p |J(x,f)| dx =
\\
K^{p + n - 2} \int_{U_2}[\widetilde{\rho}(y)]^p dy =
K^{p + n - 2} \int_{\mathbb R^n}[\widetilde{\rho}(y)]^p dy.
\end{multline}
In~(\ref{eq4.1}), we have used the fact that, since
$f$
is a $K$-quasi-isometry, it is easy to verify the inequality
$\frac{||f'(x)||^p}{|J(x,f)|} \le K^{p + n - 2}$
for
$x \in U_1 \setminus E$.
Taking~(\ref{eq4.1}) into account and recalling that the
inverse mapping
$f^{-1}$
is also
$K$-quasi-isometric, we finally get~(\ref{eq1.3}).
\end{proof}

In conclusion, note that the main results of our article were
earlier announced in~\cite{Ko4}.
\vskip3mm

{\bf Acknowledgments}
\vskip3mm

The author was partially supported by the Russian Foundation for
Basic Research (Grant 11-01-00819-a), the Interdisciplinary
Project of the Siberian and Far-Eastern Divisions of the Russian
Academy of Sciences (2012-2014 no. 56), the State Maintenance
Program for the Leading Scientific Schools of the Russian
Federation (Grant NSh-921.2012.1) and the Exchange Program
between the Russian and Polish Academies of Sciences (Project
2014-2016).

\end{document}